\documentclass[12pt]{article}%

\usepackage{amsmath,enumerate}
\usepackage{amsfonts}
\usepackage{amssymb}
\usepackage{color}

\newtheorem{theorem}{Theorem}[section]

\newtheorem{conjecture}[theorem]{Conjecture}
\newtheorem{corollary}[theorem]{Corollary}

\newtheorem{definition}[theorem]{Definition}

\newtheorem{lemma}[theorem]{Lemma}

\newtheorem{problem}[theorem]{Problem}
\newtheorem{proposition}[theorem]{Proposition}

\newenvironment{proof}[1][Proof]{\noindent\textbf{#1.} }
{\hfill \ \rule{0.5em}{0.5em}}

\newcommand{\fq}{\mathbb{F}_q}

\title{\vspace{-1cm}On the chromatic number of the Erd\H{o}s-R\'{e}nyi orthogonal polarity graph}
\author{
Xing Peng
\thanks{Department of Mathematics, University of California, San Diego, CA, USA, x2peng@math.ucsd.edu }
\and
Michael Tait
\thanks{Department of Mathematics, University of California, San Diego, CA, USA, mtait@ucsd.edu}
\and
Craig Timmons
\thanks{Department of Mathematics and Statistics, California State University Sacramento.  Research supported by CSU Sacramento Provost's Research Incentive Fund.}
}
\date{}
\begin{document}
\maketitle

\vspace{-1cm}

\begin{abstract}
For a prime power $q$, let $ER_q$ denote the Erd\H{o}s-R\'{e}nyi orthogonal polarity graph.  We prove
that if $q$ is an even power of an odd prime, then $\chi ( ER_{q}) \leq 2 \sqrt{q} + O ( \sqrt{q} / \log q)$.  This upper bound is best possible up to a constant factor of at most 2.
If $q$ is an odd  power of an odd prime and satisfies some condition on irreducible polynomials, then we improve the best known  upper bound for $\chi(ER_{q})$ substantially. We also show that for sufficiently large $q$, every $ER_q$ contains a subgraph that is
not 3-chromatic and has at most 36 vertices.
\end{abstract}

\section{Introduction}

Let $q$ be a prime power and let $V$ be a 3-dimensional vector space over $\mathbb{F}_q$.  Let $PG(2,q)$ be the projective
geometry whose points are the 1-dimensional subspaces of $V$ and whose lines are the 2-dimensional subspaces of $V$.
The \emph{Erd\H{o}s-R\'{e}nyi orthogonal polarity graph}, denoted $ER_q$, is the graph whose vertices are the points
of $PG(2,q)$, and distinct vertices $(x_0 , x_1 , x_2)$ and $(y_0 , y_1 , y_2)$ are adjacent if and only if
$x_0 y_0 + x_1 y_1  + x_2 y_2 = 0$.
One obtains an isomorphic graph if the equation for adjacency is $x_2 y_0 + x_0 y_2 = x_1 y_1$ (see
\cite{mw}) and it is this definition of $ER_q$ that we will use.

\medskip
These graphs were constructed independently by
Brown \cite{b}, and Erd\H{o}s, R\'{e}nyi, and S\'{o}s \cite{ers} and have many applications to problems in extremal graph theory.  The graph $ER_q$ has $q^2 + q + 1$ vertices, has $\frac{1}{2}q ( q+1)^2$ edges, and
has no 4-cycle as a subgraph.  F\"{u}redi \cite{fu1983}, \cite{fu1996} proved that for $q \geq 15$ a prime power,
a graph with $q^2 + q + 1$ vertices and no 4-cycle has at most $\frac{1}{2}q ( q+ 1)^2$ edges.  The Erd\H{o}s-R\'{e}nyi graphs show that F\"{u}redi's upper bound is best possible.  This is perhaps the most well-known application of $ER_q$ to extremal graph theory, but there are many others.  The interested reader is referred to \cite{lv}, \cite{kpr}, and \cite{bgm} for applications of $ER_q$ to problems in hypergraph Tur\'{a}n theory, Ramsey theory, and structural graph theory.

\medskip

Because of its important place in extremal graph theory, many researchers have studied the graph $ER_q$. Benny Sudakov posed the question of determining the independence number of $ER_q$ (see \cite{Williford Thesis}), and it has since been investigated in several papers. Mubayi and Williford \cite{mw} proved that if $p$ is a prime, $n \geq 1$ is an integer, and $q = p^n$, then
\[
\alpha (ER_q) \geq \left\{
\begin{array}{ll}
\frac{1}{2} q^{3/2} + \frac{1}{2} q + 1 & \mbox{if $p$ is odd and $n$ is even}, \\
\frac{120 q^{3/2} }{ 73 \sqrt{73} } & \mbox{if $p$ is odd and $n$ is odd}, \\
\frac{ q^{3/2} }{ 2 \sqrt{2} } & \mbox{if $p=2$ and $n$ is odd}, \\
q^{3/2} - q + q^{1/2} & \mbox{if $p =2$ and $n$ is even}.
\end{array}
\right.
\]
 An upper bound $\alpha (ER_q ) \leq q^{3/2} + q^{1/2} + 1$ can be obtained from Hoffman's bound.  Therefore, the order of magnitude of $\alpha (ER_q )$ is $q^{3/2}$. Godsil and Newman refined the upper bound obtained from Hoffman's bound in \cite{GN}. Their result was then improved using the Lov\'{a}sz theta function in \cite{dNPS}. When $q$ is even, Hobart and Williford \cite{HW} used coherent configurations to provide upper bounds for the independence number of general orthogonal polarity graphs. When $q$ is an even square,  the know upper bound and lower bound for  $\alpha(ER_q)$ differ by at most $1$. In the case when $p$ is odd or when $p = 2$ and $n$ is odd, it is still an open problem to determine an asymptotic formula for
$\alpha ( ER_q)$.

\medskip
 Since the independence number has been well-studied and its order of magnitude is known, it is  natural  to  investigate the chromatic number of $ER_q$ which is closely related to $\alpha(ER_q)$. As $ER_q$ has $q^2+q+1$ vertices  and $\alpha(ER_q)=\Theta(q^{3/2})$, a  lower bound for $\chi(ER_q)$ is $\frac{q^2 + q + 1}{ \alpha (ER_{q} ) } \geq q^{1/2}$. One may ask whether this  lower bound actually gives the right order of magnitude of $\chi(ER_q)$. We confirm this for $q$ being an even power of an odd prime.
\medskip

%The graph $ER_q$ has maximum degree $q + 1$ so by Brooks Theorem, $\chi ( ER_q) \leq q$.  This can be improved to
%$\chi (ER_q) = O ( q / \log q )$ using a result of Alon, Krivelevich, and Sudakov \cite{aks}.
%Based on Mubayi's and Williford's estimates on the independence number of $ER_q$ and the symmetries of $ER_q$, it should be the case that $\chi (ER_q)$ is closer to $\frac{q^2 + q + 1}{ \alpha (ER_q) } = O ( q^{1/2})$.
%The main purpose of this work is to improve the known upper bounds on $\chi (ER_q)$.

\begin{theorem}\label{th:chromatic number}
If $q = p^{2r}$ where $p$ is an odd prime and $r \geq 1$ is an integer, then
\[
\chi ( ER_q ) \leq 2 \sqrt{q}+ O ( \sqrt{q} / \log q ).
\]
\end{theorem}

This upper bound is within a factor of 2 of the lower bound $\chi (ER_{q} ) \geq \frac{q^2 + q + 1}{ \alpha (ER_{q} ) } \geq q^{1/2} $.  Any improvement in the coefficient of $q^{1/2}$ would give an improvement to the best known lower bound on the independence number of $ER_q$ from \cite{mw}.   The lower order term $O( q^{1/2} / \log q )$ is obtained using probabilistic methods \cite{aks} and while the implied constant is absolute, we have not made an effort to compute it.  By using Brooks' Theorem instead of
the result of Alon et.\ al. \cite{aks}, we obtain the upper bound $\chi (ER_q) \leq 4q^{1/2}+1$ for all $q = p^{2r}$ where $r\geq 1$ is an integer and $p$ is an odd prime.

When $q$ is not an even power, we first prove the following general theorem.

\begin{theorem}\label{th:chromatic number 2}
Let $q$ be an odd power of an odd prime and let $r \geq 1$ be an integer.  If there is a $\mu \in \mathbb{F}_q$
such that $x^{2r+1} - \mu$ is irreducible in $\mathbb{F}_q [x]$, then
\[
\chi ( ER_{q^{2r+1} } ) \leq  \frac{2r+5}{3} q^{ \frac{4r}{3}+1 }+(2r+1)q^{r+1}+1.
\]
\end{theorem}

Given an odd integer $2r + 1 \geq 3$, there are infinitely many primes $p$ for which there is a $\mu \in \fq$ such that
$x^{2r+1} - \mu \in \fq [x]$ is irreducible (see Section 5 for more details), where $q$ is an arbitrary odd power of $p$.   Our method can also be used to prove
that if $q$ is a power of any odd prime, then
\[
\chi( ER_{q^3} ) \leq 6q^2 + 1.
\]
Here we do not need the existence of an irreducible polynomial $x^3 - \mu \in \fq [x]$.

For $q$ is an odd power of an odd prime, we have the following corollary.

\begin{corollary}
Let $q=p^s$ for an odd prime $p$ and an odd integer $s \geq 3$.  If $t>1$ is the smallest divisor of $s$ such that $x^t-\mu$ is irreducible in $\mathbb F_{p^{s/t}}[x]$  for some  $\mu \in {\mathbb F}_{p^{s/t}}$, then
\[
\chi(ER_q) \leq  \frac{t+4}{3} p^{s(2t+1)/3t}+tp^{(t+1)/2}+1.
\]

\end{corollary}

 We encountered difficulties in extending this upper bound to the general case.  In particular, when $p$ is a prime, we have not been able to improve the upper bound
$\chi (ER_p) = O ( p / \log p )$ which is obtained by applying the main result of \cite{aks}.

\begin{conjecture}
Let $p$ be an odd prime.  For any integer $r \geq 0$,
\[
\chi ( ER_{ p^{2r+1} } ) = O( p^{r + 1/2} ).
\]
\end{conjecture}

Instead of working with $ER_q$, we work with a related graph that is a bit more suitable for our computations.

\begin{definition}
Let $q$ be a power of an odd prime and $A = \{ ( a, a^2) : a \in \mathbb{F}_q \}$.  Let $G_q$ be the graph
with vertex set $\fq \times \fq$, and distinct vertices $(x_1 , x_2)$ and $(y_1 , y_2)$ are adjacent if and only if
\[
(x_1 , x_2) + (y_1 , y_2 ) \in A.
\]
\end{definition}

Let $G_q^{\circ}$ be the graph obtained from $G_q$ by adding loops to all vertices $(x_1 , x_2)$ for which
$(x_1 , x_2) + (x_1 , x_2) \in A$.
Vinh \cite{vinh} proved that the graph $G_q^{\circ}$ is a $(q^2 , q , \sqrt{2q} )$-graph.
Recall an $(n,d, \lambda)$ graph is an $n$-vertex $d$-regular graph whose second eigenvalue
$\max \{ | \lambda_2 | , | \lambda_n | \}$ is at most $\lambda$.
Vinh used the fact that $G_q^{\circ}$ is a $(q^2 , q , \sqrt{2q} )$-graph to count solutions to
$x_1 + x_2 = (x_3 + x_4)^2$ where $(x_1 , x_3) \subset B$, $(x_2 , x_4) \subset C$, and
$B,C \subset \fq^2$.  For similar results that are obtained using techniques from combinatorial number theory,
see \cite{cill}.  We prove that $G_q$ is isomorphic to an induced subgraph of the Erd\H{o}s-R\'{e}nyi orthogonal polarity graph.

\begin{theorem}\label{th:gq}
If $q$ is a power of an odd prime, then the graph $G_q$ is isomorphic to an induced subgraph of $ER_q$.
\end{theorem}

In the course of proving Theorem \ref{th:gq} we will show how to obtain $ER_q$ from $G_q$ by adding vertices and edges
to $G_q$.  This will allow us to translate upper bounds on $\chi (G_q)$ to upper bounds on $\chi (ER_q)$.

In addition to finding a proper coloring of $ER_q$, we also investigate proper colorings of small subgraphs of $ER_q$.
In particular, we obtain the following result concerning small subgraphs of $ER_q$ that are not 3-colorable.

\begin{theorem}\label{th:main theorem}
If $q$ is sufficiently large, then $ER_q$ contains a subgraph $H$ with at most 36 vertices and
$\chi (H) \geq 4$.
\end{theorem}

If $\mathcal{F}$ is a family of graphs, we say that a graph $G$ is $\mathcal{F}$-free if $G$ does not contain a subgraph isomorphic to a graph in $\mathcal{F}$.  One of the most well-studied problems in extremal graph theory is to determine the maximum number of edges in an $\mathcal{F}$-free graph with $n$ vertices, and the describe the extremal $\mathcal{F}$-free graphs.
When $\mathcal{F}$ contains bipartite graphs, the structure of extremal $\mathcal{F}$-free graphs is not very well understood and there are few general results.  A notable exception is when $\mathcal{F} = \{ C_4 \}$.  A result of F\"{u}redi \cite{fu1996} states that when $q \geq 15$ is a prime power, an extremal $C_4$-free graph with $q^2 + q + 1$ vertices is an orthogonal polarity graph of a projective plane (see Section 5 for the definition of an orthogonal polarity graph).
%Despite this result, there are still open problems concerning extremal $C_4$-free graphs.  For instance Erd\H{o}s and Simonovits %\cite{es} conjectured that any $n$-vertex graph with $\textup{ex}(n , C_4) + 1$ edges must contain at least $n^{1/2} + o ( %n^{1/2} )$ 4-cycles but this is still open.

Let $\mathcal{C}^r$ be the family of graphs with chromatic number $r$, and
$\mathcal{C}_k^r$ be the family of graphs with at most $k$ vertices and chromatic number $r$.
Theorem \ref{th:main theorem} is motivated by the following problem of Allen, Keevash, Sudakov, and Verstra\"{e}te \cite{aksv}.

\begin{problem}[Allen, et al.\ \cite{aksv}]\label{problem 1}
Let $\mathcal{F}$ be a family of bipartite graphs.  Determine if there is an integer $k$ such that
\[
\textup{ex}( n , \mathcal{F} \cup \mathcal{C}_k^r ) \sim \textup{ex}(n , \mathcal{F} \cup \mathcal{C}^r ).
\]
\end{problem}

When considering Problem \ref{problem 1}, a question that arises is if every extremal $\mathcal{F}$-free $n$-vertex graph (here $n$ is tending to infinity) must contain some member of $\mathcal{C}_k^r$?  In other words, does forbidding $\mathcal{C}_k^r$ actually have an effect on extremal $\mathcal{F}$-free graphs.
By Theorem \ref{th:main theorem}, one cannot take $ER_q$ to obtain a lower bound on the Tur\'{a}n number
$\textup{ex}( n , \{C_4 \} \cup \mathcal{C}_k^3 )$ for $k \geq 36$ without modifying $ER_q$ in some way.
It seems likely that for any integer $r \geq 5$, there exists integers $q_r$ and $f(r)$ such that for any $q \geq q_r$, the graph $ER_q$ contains a subgraph with at most $f(r)$ vertices and chromatic number at least $r$.

In Section 2 we prove Theorem \ref{th:gq}.  In Section 3 we prove Theorems \ref{th:chromatic number} and
\ref{th:chromatic number 2}.  In Section 4 we prove Theorem \ref{th:main theorem}.  Section 5 contains some concluding remakrs.

%%%%%%%%%%%%%%%%%%%%%%%%%%%%%%%%%%%%%%%%%%%%%%%%%%%%%%%

\section{Proof of Theorem \ref{th:gq}}

Let $q$ be a power of an odd prime power and $A = \{ ( a, a^2) : a \in \mathbb{F}_q \}$.  Let $\fq = \{ b_1 , \dots , b_q \}$ and assume that
$b_q = 0$.  Let $F = \{ b_q \} \times \fq$.  Then $F$ is a subgroup of $\fq^2$ and we let
\[
F_i = F + (b_i , 0)
\]
be the cosets of $F$ where $F_q = F$.  Add new vertices $z_1 , \dots , z_q , y$ to $G_q$.  Make
$z_i$ adjacent to all vertices in $F_i$, and make $y$ adjacent to each $z_i$.  Call this graph $H_q$.  Observe that $G_q$ is an induced subgraph of $H_q$.  We define an isomorphism $\phi$ from $H_q$ to $ER_q$ as follows.
\begin{enumerate}
\item For any $b_j \in \fq$, let $\phi ( (0,b_j)) = ( 1 , 0 , 2^{-1} b_j )$.
\item For any $b_i , b_j \in \fq$ with $b_i \neq 0$, let $\phi ( (b_i , b_j) ) = (1 , b_i , 2^{-1} ( b_j - b_i^2 ) )$.
\item Let $\phi (y) = (0,0,1)$ and $\phi ( z_i)  = (0,1,b_i)$ for $1 \leq i \leq q$.
\end{enumerate}

We will show that $\phi$ is an isomorphism by considering the different types of vertices in $H_q$.  Recall that the rule for
adjacency in $ER_q$ is that $(x_0 , x_1 , x_2)$ is adjacent to $(y_0 , y_1 , y_2)$ if and only if
$x_0 y_2 + x_2 y_0 = x_1 y_1$.

\noindent
\textbf{Case 1}: Vertices of type $(0,b_j)$.

Let $b_j \in \fq$.  In $H_q$, the neighborhood of $(0,b_j)$ is
$\{z_q \} \cup \{ ( x , x^2 - b_j ) : x \in \fq \}$.
In $ER_q$, the neighborhood of $(1, 0 , 2^{-1} b_j)$ is
\begin{equation}\label{iso eq 1}
\{ (0,1,0) \} \cup \{ ( 1 , x , - 2^{-1} b_j) : x \in \fq \}.
\end{equation}
By definition, $\phi ( ( 0 , -b_j) ) = (1 , 0 , -2^{-1} b_j )$ and for $x \neq 0$,
\[
\phi ( (x , x^2 - b_j ) ) = (1 , x  , 2^{-1}( x^2 - b_j - x^2 ) ) = ( 1 , x , -2^{-1} b_j ).
\]
This shows that (\ref{iso eq 1}) coincides with the set
\[
\{ \phi ( z_q ) \} \cup \{ \phi ( ( x , x^2 - b_j ) ) : x \in \fq \}.
\]
We conclude that for any $b_j \in \fq$, $(0,b_j)$ is adjacent to $u$ in $H_q$ if and only if
$\phi ( ( 0 , b_j) )$ is adjacent to $\phi (u)$ in $ER_q$.

\smallskip

\noindent
\textbf{Case 2}: Vertices of type $(b_i , b_j)$ with $b_i \neq 0$.

Let $b_i , b_j \in \fq$ with $b_i \neq 0$.  In $H_q$, the neighborhood of $(b_i , b_j)$ is
\[
\{ z_i  \} \cup \{ ( x - b_i , x^2 - b_j ) : x \in \fq \}.
\]
In $ER_q$, the neighborhood of $(1 , b_i , 2^{-1} ( b_j - b_i^2 ) )$ is
\begin{equation}\label{iso eq 2}
\{ (0,1,b_i) \} \cup \{ ( 1, x, xb_i - 2^{-1} ( b_j - b_i^2 ) ) : x \in \fq \}.
\end{equation}
We have $\phi (z_i) = (0,1,b_i)$ and
\[
\phi ( (b_i - b_i , b_i^2 - b_j )) = ( 1 , 0 , - 2^{-1} ( b_j - b_i^2 ) ).
\]
For $y \neq b_i$,
\begin{eqnarray*}
\phi ( ( y -b_i , y^2 - b_j ) ) & = & (1 , y - b_i  , 2^{-1} (y^2 - b_j - (y - b_i)^2 ) ) \\
& = & ( 1 , y - b_i , yb_i - 2^{-1} ( b_j + b_i^2) ).
\end{eqnarray*}
If we take $x = y - b_i$ in (\ref{iso eq 2}), we obtain
\[
( 1 , y - b_i , yb_i - 2^{-1} ( b_j +b_i^2) )
\]
using the fact that $2^{-1}  - 1 = - 2^{-1}$.  We conclude that for any $b_i , b_j \in \fq$ with $b_i \neq 0$,
$(b_i  ,b_j)$ is adjacent to $u$ in $H_q$ if and only if $\phi ( ( b_i , b_j) )$ is adjacent to $\phi (u)$ in $ER_q$.

\smallskip

\noindent
\textbf{Case 3}: Vertices of type $z_i$.

Let $1 \leq i \leq q$ and consider $z_i$.  The neighborhood of $z_i$ is $
\{ y \} \cup \{ (b_i , x) : x \in \fq \}$.
In $ER_q$, the neighborhood of $(0,1,b_i)$ is
\[
\{ (0,0,1) \} \cup \{ ( 1 , b_i , x ) : x \in \fq \} = \{ \phi (y) \} \cup \{ (1,b_i , x) : x \in \fq \}.
\]
If $i = q$, then $\phi ( (0,y) ) = (1 , 0 , 2^{-1} y )$.  If $i \neq q$, then
$\phi ((b_i, y) ) = (1 , b_i , 2^{-1} (y - b_i^2 ) )$.  As $y$ ranges over $\fq$, we obtain
$(1 , b_i ,   x )$ for all $x \in \fq$.

\smallskip
We have not checked the neighborhood condition for $y \in V(H_q )$ but since we have considered all other vertices, this is not necessary.

%%%%%%%%%%%%%%%%%%%%%%%%%%%%%%%%%%%%%%%%%%%%%%%%%%%%%%%%%

\section{Proof of Theorems \ref{th:chromatic number} and \ref{th:chromatic number 2} }

Throughout this section $p$ is an odd prime and $q$ is a power of $p$.  The set
$\fq^*$ consisting of the nonzero elements of $\fq$ can be partitioned into two sets
$\fq^+$ and $\fq^-$ where
\begin{center}
$a \in \fq^+$ if and only if $-a \in \fq^-$.
\end{center}
Observe that the vertices $(x_1 , x_2)$ and $(y_1 , y_2)$ are adjacent in $G_q$ if and only if 
$x_1 + y_1 = a$ and $x_2 + y_2 = a^2$ for some $a \in \fq$.  This is equivalent to
$(x_1 + y_1)^2 = x_2 + y_2$.  It is often this relation that we will use in our calculations.

\begin{lemma}\label{le:6.1 lemma}
(i) If $\mathbb{F}_{q^2} = \{ a \theta + b : a , b \in \fq \}$ for some $\theta \in \mathbb{F}_{q^2}\setminus \mathbb{F}_q$, then both
\begin{center}
$\{ ( x , y \theta + z) : x,z \in \fq , y \in \fq^+ \}$ and
$\{ (x , y  \theta + z) : x,z \in \fq , y \in \fq^- \}$
\end{center}
are independent sets in $G_q$.

\noindent
(ii) If $t \geq 3$ is odd and $\mathbb{F}_{q^t} = \{ a_0 + \dots + a_{t-1} \theta^{t-1} : a_i \in \fq \}$ for some
$\theta \in \mathbb{F}_{q^t}$, then both
\[
\{ (x_0 + \dots + x_{ (t-3)/2 } \theta^{ \frac{t-3}{2} } , y_0  + \dots + y_{t-1} ) : x_i , y_j \in \fq , y_{t-1} \in \fq^+ \}
\]
and
\[
\{ (x_0 + \dots + x_{ (t-3)/2 } \theta^{ \frac{t-3}{2} } , y_0  + \dots + y_{t-1} ) : x_i , y_j \in \fq , y_{t-1} \in \fq^- \}
\]
are independent sets in $G_q$.
\end{lemma}
\begin{proof}
We prove the first case of (i) as the proofs of the remaining statements are very similar.  Suppose
$(x_1 , y_1 \theta + z_1)$ and $(x_2 , y_2 \theta + z_2 )$ are vertices in $G_q$ with
$x_1 , x_2 , z_1 , z_2 \in \fq$ and $y_1 , y_2 \in \fq^+$.  Then
$(x_1 + x_2)^2 \in \fq$ but $(y_1 + y_2) \theta + (z_1 + z_2) \notin \fq$ since
$y_1 + y_2 \neq 0$.  Therefore, the vertices $(x_1 , y_1 \theta + z_1)$ and $(x_2 , y_2 \theta + z_2 )$ are not adjacent.
\end{proof}

\begin{lemma}\label{le:6.2 lemma}
For any $k \in \fq^*$, the maps $\psi_k , \phi_k : V( G_q ) \rightarrow V(G_q )$ given by
\begin{center}
$\psi_k ( (x,y) ) = ( x+ k , y  + 4kx + 2k^2)$ and $\phi_k ((x,y) )  = (kx , k^2y)$
\end{center}
are automorphisms of $G_q$.
\end{lemma}
\begin{proof}
Let $k \in \mathbb{F}_{q}^*$.  Suppose $(x_1 , x_2)$ is adjacent to $(y_1 , y_2)$ so that
$(x_1 + y_1)^2 = x_2 + y_2$.  In this case,
\begin{eqnarray*}
(x_1 + k + y_1 + k )^2 & = & (x_1 + y_1)^2 + 4kx_1 + 4k y_1 + 4k^2 \\
& = & ( x_2 + 4kx_1 + 2k^2) + (y_2 + 4k y_1 + 2k^2).
\end{eqnarray*}
This shows that $(x_1 + k , x_2 + 4kx_1 + 2k^2)$ is adjacent to $(y_1 + k , y_2 + 4ky_1 + 2k^2)$.  Conversely,
if $(x_1 + k , x_2 + 4kx_1 + 2k^2)$ is adjacent to $( y_1 + k, y_2 + 4ky_1 + 2k^2)$, then it must be the case that
$(x_1 + y_1)^2 = x_2 + y_2$ and so $(x_1 , y_1)$ is adjacent to $(x_2 , y_2)$.

To show that $\phi_k$ is an isomorphism it is enough to observe that
$(x_1  + y_1)^2 = x_2 + y_2$ is equivalent to $( kx_1 + ky_1)^2 = k^2 x_2 + k^2 y_2$.
\end{proof}

%%%%%%%%%%%%%%%%%%%%%%%%%%%%%%%%%%%%%%%%%%%%%%%%%%%%

\subsection{$q$ a square}

In this section we prove the following.

\begin{theorem}\label{square chrom}
Let $q$ be a power of an odd prime.  The chromatic number of $G_{q^2}$ satisfies
\[
\chi ( G_{q^2} ) \leq 2 q + O( q / \log q) .
\]
\end{theorem}
\begin{proof}
Let $\theta$ be a root of an irreducible quadratic polynomial in $\fq [x]$ so that
$\mathbb{F}_{q^2} = \{ a \theta + b : a , b, \in \fq \}$.  Assume that
$\theta^2 = \mu_1 \theta + \mu_0$ where $\mu_0 , \mu_1 \in \fq$.
Let $I^+ = \{ ( x , y \theta + z) : x , z \in \fq , y \in \fq^+ \}$,
$I^- = \{ (x , y \theta + z) : x , z \in \fq , y \in \fq^- \}$, and $J = I^+ \cup I^-$.  By Lemma
\ref{le:6.1 lemma}, $J$ is the union of two independent sets and so $\chi ( G_{q^2} [J] ) \leq 2$.
Let
\[
S = \bigcup_{ k \in \fq } \psi_{ k \theta } (J).
\]
By Lemma \ref{le:6.2 lemma}, each $\psi_{k \theta }$ is an isomorphism and so
$\chi( G_{q^2} [S] ) \leq 2q$.  Let $X = V(G_{q^2} ) \backslash S$.  Since
$\fq^+ \cup \fq^- = \fq^*$, we can write
\[
S = \{ ( x + k \theta , y \theta + z +  4k \theta x + 2 k^2 \theta^2 ) : x , k , y , z \in \fq , y \neq 0 \}.
\]
Given a vertex $(s,t) \in V(G_{q^2} )$, say with $s = s_0 + s_1 \theta$ and $t = t_0 + t_1 \theta $, we can take
$x = s_0$ and $k = s_1$ to obtain
\[
\{ (s_0 + s_1 \theta , y \theta + z + 4s_1 s_0 \theta + 2s_1^2 ( \mu_1 \theta + \mu_0 ) ) : y,z \in \fq , y \neq 0 \} \subset S.
\]
The second coordinate in the above subset of $S$ simplifies to
\[
(z + 2s_1^2 \mu_0 ) + ( y + 4s_1 s_0 + 2s_1^2 \mu_1) \theta.
\]
We can choose $z = t_0 - 2s_1^2 \mu_0$ and as long as
$t_1 \neq 4s_1 s_0 + 2s_1^2 \mu_1$, we can take
$y = t_1 - 4s_1 s_0 - 2s_1 ^2 \mu_1$.  Otherwise, $t_1 = 4s_1 s_0 + 2s_1^2 \mu_1$ and so
\[
X = \{ (s_0 + s_1 \theta , t_0 + (4s_1 s_0  + 2s_1^2 \mu_1) \theta ) : s_0 , s_1 , t_0 \in \fq \}.
\]
Partition $X$ into $q$ sets $X_s$ where $s \in \fq$ and
\[
X_{s} = \{ ( s \theta + s_2 , (2s^2 \mu_1 + 4s s_2 ) \theta + t_2 : s_2 , t_2 \in \fq \}.
\]

\smallskip

\noindent
\textbf{Claim 1}: For any $s \in \fq$, $\ \Delta ( G_{q^2} [ X_s ] ) \leq q$.

\smallskip

Let $s \in \fq$.  A pair of vertices
\begin{center}
$( s \theta + s_2 , (2s^2 \mu_1 + 4s s_2) \theta + t_2 )$ and
$( s \theta + u_2 , (2s^2 \mu_1 + 4s u_2) \theta + v_2)$,
\end{center}
both in $X_s$, are adjacent if and only if
$4s^2 \mu_2 + (s_2 + u_2)^2 = t_2 + v_2$.  If $s_2$ and $t_2$ are fixed, then there are $q$ choices for $u_2$ and
once $u_2$ is fixed, $v_2$ is determined.  Therefore, the maximum degree of $G_{q^2}[X_s]$ is $q$.

\smallskip

\noindent
\textbf{Claim 2:} $\Delta ( G_{q^2} [X] ) \leq 2q-1$.

\smallskip

By Claim 1, a vertex in $X_s$ has at most $q$ other neighbors in $X_s$.
Let $s ,t \in \fq$ where $s \neq t$.  The vertex
$(s \theta + s_2 , (2 s^2 \mu_1 + 4s s_2 ) \theta + t_2) \in X_s$ is adjacent to the vertex
$(t \theta + u_2 , (2 t^2 \mu_1 + 4t u_2 ) \theta + v_2) \in X_t$ if and only if
\begin{equation}\label{claim 4 eq 1}
\mu_1 (s^2 + 2st + t^2) + 2 (s + t) s_2 + 2 (s + t) u_2 = \mu_1 ( 2s^2  + 2t^2) + 4s s_2 + 4t u_2
\end{equation}
and
\begin{equation}\label{claim 4 eq 2}
(s+t)^2 \mu_0 + (s_2 + u_2)^2 = t_2 + v_2.
\end{equation}
Equation (\ref{claim 4 eq 1}) can be rewritten as
\begin{equation}\label{claim 4 eq 3}
\mu_1 ( s- t)^2 = 2 (t - s) s_2 + 2(s - t) u_2.
\end{equation}
Thus if $s_2$ and $t_2$ are fixed, then (\ref{claim 4 eq 3}) and (\ref{claim 4 eq 2}) determine $u_2$ and $v_2$
since $2 (  s - t ) \neq 0$.  This shows that
a vertex in $X_s$ has exactly one neighbor in $X_t$ whenever $s \neq t$.  Namely, given
the vertex $(s \theta + s_2 , (2 s^2 \mu_1 + 4s s_2  ) \theta + t_2) \in X_s$, its unique neighbor
in $X_t$ where $t \neq s$ is $(t \theta + u_2 , (2 t^2 \mu_1 + 4t u_2 ) \theta + v_2)$ where
\begin{center}
$u_2 = 2^{-1} \mu_1 (s - t) + s_2$ and $v_2 = (s + t)^2 \mu_0 + (2s_2 + 2^{-1} \mu_1 ( s - t))^2 - t_2$.
\end{center}
We conclude that a vertex $x \in X$ has at most $q$ neighbors in $X_s$ when $x \in X_s$, and one neighbor in each
$X_t$ for $t \neq s$.  Since $X = \cup_{s \in \fq } X_s$, we have proved Claim 2.

\smallskip

Alon, Krivelevich, and Sudakov \cite{aks} proved that any graph with maximum degree $d$ with the property that the neighborhood of every vertex contains at most $d^2 / f $ edges has chromatic number at most
$c ( d / \log f)$ where $c$ is an absolute constant.  A $C_4$-free graph with maximum degree $d$ has the property that the neighborhood of every vertex contains at most $d / 2 $ edges.  Applying the result of \cite{aks} to $G_{q^2}[X]$, we obtain
$ \chi ( G_{q^2}[X] ) = O( 2q  / \log q )$.
Combining this coloring with our coloring of $S$, we obtain a proper coloring of $G_{q^2}$ with $2q + O( q / \log q)$ colors.
\end{proof}

\smallskip

To obtain a coloring of $ER_{q^2} \cong H_{q^2}$, we only need one additional color for the vertices
$z_1 , \dots , z_{q^2},y$.  The vertices $z_1 , \dots , z_{q^2}$ form an independent set in $H_{q^2}$ and so we use one new color on these vertices.  The vertex $y$ has no neighbors in $G_{q^2}$ and so we may use any one of the $2q + O ( q /  \log q )$ colors used to $G_{q^2}$ to color $y$.  This proves Theorem \ref{th:chromatic number}.

%%%%%%%%%%%%%%%%%%%%%%%%%%%%%%%%%%%%%%%%%%%%%%%%%%%%

\subsection{$q$ not a square}

In this subsection we prove the following result.
\begin{theorem}\label{th:gq coloring}
Let $q$ be a power of an odd prime.  If $r \geq 1$ and for some $\mu \in \fq$, the polynomial
$x^{2r+1} - \mu \in \fq [x]$ is irreducible, then
\[
\chi ( G_{q^t} ) \leq \frac{2r+5}{3} q^{ \frac{4r}{3}+1 }+(2r+1)q^{r+1}.
\]
\end{theorem}
\begin{proof}
%Let $t = 2r  + 1 \geq 3$ be an odd integer.
Suppose there is a
$\mu \in \fq^*$ such that the polynomial $x^t - \mu \in \fq [x]$ is irreducible.  Let $\theta$ be a root of $x^t - \mu$ in an extension field of $\fq$.  We may view $\theta$ as an element of $\mathbb{F}_{q^t}$ and $\{1, \theta , \dots , \theta^{2r} \}$ is a basis for
$\mathbb{F}_{q^t}$ over $\fq$.
For $2r+1 \leq l \leq 6r+3$,
\[
\theta^l =
\left\{
\begin{array}{ll}
\mu \theta^{l - 2r-1} & \mbox{if $2r+1 \leq l < 4r+2$}, \\
\mu^2 \theta^{l - 4r-2} & \mbox{if $4r + 2 \leq l < 6r + 3$}.
\end{array}
\right.
\]
This identity will be used frequently throughout this subsection.
Define
\[
I^+ = \{ ( x_0 + x_1 \theta + \dots + x_{r-1} \theta^{r-1} , y_0 + \dots + y_{2r} \theta^{2r} ) : x_i , y_j \in \fq, y_{2r} \in \fq^+ \}
\]
and
\[
I^- = \{ ( x_0 + x_1 \theta + \dots + x_{r-1} \theta^{r-1} , y_0 + \dots + y_{2r} \theta^{2r} ) : x_i , y_j \in \fq, y_{2r} \in \fq^- \}.
\]
By Lemma \ref{le:6.1 lemma}, both $I^+$ and $I^-$ are independent sets.
Let $J = I^+ \cup I^-$.  Since $J$ is the union of two independent sets, $\chi ( G_{q^t} [J] ) \leq 2$.  For
$k \in \mathbb{F}_{q^t}$, the map $\psi_{k} ( (x,y) ) = (x + k , y + 4kx + 2k^2)$ is an isomorphism of $G_{q^t}$
by Lemma \ref{le:6.2 lemma}.  Let
\[
S = \bigcup_{ (x_r , \dots , x_{2r} ) \in \fq^{r+1} } \psi_{ x_r \theta^r + \dots + x_{2r} \theta^{2r} } (J).
\]
We properly color the vertices of $S$ with at most $2q^{r+1}$ colors.  Let $X = \mathbb{F}_{q^t}^2 \backslash S$.  It remains to color the
vertices in $X$.  To do this, we will proceed as follows.  By Lemma \ref{le:6.2 lemma}, for any $k \in \mathbb{F}_{q^t}^*$, the map
$\phi_{k} ( ( x, y) ) = (kx , k^2 y )$ is an isomorphism of $G_{q^t}$.  Let $1 \leq l \leq 2r$ and consider
$\phi_{ \theta^l }(X)$.  Let $Y_l = S \cap \phi_{ \theta^l} (X)$.  The
graph $G_{q^t}[ Y_l]$ is isomorphic to a subgraph of $G_{q^t} [S]$.  We have shown that
$\chi ( G_{q^t} [S] ) \leq 2q^{r+1}$ and so $\chi ( G_{q^t} [Y_l] ) \leq 2q^{r + 1}$ for any $1 \leq l \leq 2r$.
Therefore, we can properly color the vertices in
\[
\phi_{ \theta^l}^{-1} (Y_l) = \phi_{ \theta^{l}}^{-1} (S) \cap X
\]
with at most $2q^{r+1}$ colors.  This gives a proper coloring that uses at most $(2r+1) 2q^{r+1}$ colors.  The only vertices that have not been colored are those that are in the set
\[
Z:= X \cap \phi_{ \theta} (X) \cap \phi_{ \theta^2}(X) \cap \dots \cap \phi_{ \theta^{2r} } (X).
\]
We are now going to show that if $(s_0 + \dots + s_{2r} \theta^{2r} , t_0 + \dots + t_{2r} \theta^{2r} ) \in Z$, then
each $t_i$ is determined by $s_0 , \dots, s_{2r}$.  This will allow us to prove an upper bound
on the maximum degree of $G_{q^t}[Z]$ and we can then color $G_{q^t}[Z]$ by applying Brooks' Theorem.

We will use the following notation for the rest of this section.  If $s \in \mathbb{F}_{q^t}$, then $s_0 , \dots , s_{2r}$ will be the coefficients of $s$ in the unique representation
$s = s_0 + s_1 \theta + \dots + s_{2r} \theta^{2r}$ where $s_i \in \fq$.
Given a $2r+1$-tuple $(z_0 , z_1 , \dots , z_{2r} ) \in \fq^{2r+1}$, define
\[
\alpha ( z_0 , z_1 , \dots , z_{2r} ) = 2z_r^2 + 4 \sum_{j = 0}^{r - 1} z_j z_{2r - j}.
\]

\medskip

\noindent
\textbf{Claim 1}: If $(s_0 + \dots + s_{2r} \theta^{2r} , t_0 + \dots + t_{2r} \theta^{2r} ) \in X$, then
\[
t_{2r} = \alpha ( s_0 , s_1 , \dots , s_{2r} ).
\]

\smallskip

\begin{proof}[Proof of Claim 1]
A vertex in $S$ is of the form

\smallskip

\noindent
~~~~~$(x_0 + \dots + x_{r-1} \theta^{r-1} + x_r \theta^r  + \dots + x_{2r} \theta^{2r} , y_0 + \dots + y_{2r} \theta^{2r}$

\smallskip
\noindent
~~~~~~~~~~$+ 4 (x_r \theta^r + \dots  + x_{2r} \theta^{2r} )( x_0 + \dots + x_{r-1} \theta^{r-1} ) + 2(x_r \theta^r + \dots +
x_{2r} \theta^{2r} )^2 )$

\smallskip

\noindent
for some $x_i , y_j \in \fq$, and $y_{2r} \in \fq^+ \cup \fq^- = \fq^*$.  The coefficient of $\theta^{2r}$ in the second coordinate
is
\[
y_{2r} + 2x_r^2 + 4 \sum_{j=0}^{r-1} x_j x_{2r - j}.
\]
Thus, given any vertex $(s,t) \in \mathbb{F}_{q^t}^2$, we have that $(s,t) \in S$ unless
\[
t_{2r} = 2s_r^2 + \sum_{j=0}^{r-1} s_j s_{2r - j}.
\]
\end{proof}

\medskip

\noindent
\textbf{Claim 2}: If $1 \leq l \leq 2r$ and $(s,t) \in X \cap \phi_{ \theta^l} (X)$, then
\[
t_{2l-1} = \mu \alpha (s_l , s_{l+1} , \dots , s_{2r} , \mu^{-1} s_0 , \dots , \mu^{-1} s_{l-1} ) ~~ \mbox{if $1 \leq l \leq r$},
\]
and
\[
t_{2l - 2r - 2}  = \mu^2 \alpha (s_l , s_{l+1} , \dots , s_{2r} , \mu^{-1} s_0 , \dots , \mu^{-1} s_{l-1} ) ~~ \mbox{if $r+1 \leq l \leq 2r$}.
\]

\smallskip
\begin{proof}[Proof of Claim 2]
Suppose $(s,t) \in X \cap \phi_{ \theta^l }(X)$.  There is an $(x,y) \in X$ such that
$(s,t) = \phi_{ \theta^l }((x,y))$.  From the equation $(s,t) = ( \theta^l x , \theta^{2l} y )$ we obtain by equating coefficients of
$\theta^0 , \theta^1 , \dots , \theta^{2r}$ in the first component,
\begin{equation}\label{eq:1.2}
x_i = s_{l + i} ~ \mbox{for $0 \leq i \leq 2r - l$ ~and} ~~\mu x_i = s_{i - 2r + l - 1} ~\mbox{for $2r - l + 1 \leq i \leq 2r$}.
\end{equation}
If $1 \leq l \leq r$, then we obtain $t_{2l-1} = \mu y_{2r}$ by considering the coefficient of $\theta^{2l-1}$ in the second component.
Similarly, if $r+1 \leq l \leq 2r$, we obtain $t_{2l - 2r  - 2} = \mu^2 y_{2r}$ by considering the coefficient of
$\theta^{2l - 2r - 2}$ in the second component.
Since $(x,y) \in X$, we have by Claim 1 that
\begin{equation}\label{eq:1.3}
y_{2r} = \alpha (x_0 , x_1 , \dots , x_{2r} ).
\end{equation}
Using (\ref{eq:1.2}), we can solve for the $x_i$'s in terms of the $s_j$'s and then substitute into (\ref{eq:1.3}) to complete the proof of Claim 2.
\end{proof}

\medskip

For $0 \leq k \leq 2r$, let
\[
U_k = \{ \{i,j \} \subset \{0,1, \dots , 2r \} : i + j \equiv k ( \textup{mod}~2r+1 ) \}.
\]
Given $\{i,j \} \subset \{0,1, \dots , 2r \}$, let
\[
\mu_{ \{ i , j \} } =
\left\{
\begin{array}{ll}
1 & \mbox{if $1 \leq i + j \leq 2r$}, \\
\mu & \mbox{if $2r+1 \leq i + j \leq 4r-1$}.
\end{array}
\right.
\]

\medskip

\noindent
\textbf{Claim 3}: Suppose $(s,t) \in Z$.  If $1 \leq l \leq r$, then
\[
t_{2l-1} = 2 \mu s_{l+r}^2 + 4 \sum_{ \{i,j \} \in U_{2l-1} } \mu_{ \{i,j \} } s_i s_j.
\]
If $0 \leq l \leq r - 1$, then
\[
t_{2l} = 2 s_l^2 + 4 \sum_{ \{i,j \} \in U_{2l} } \mu_{ \{i,j \} } s_i s_j.
\]

\smallskip

\begin{proof}[Proof of Claim 3]
First suppose $1 \leq l \leq r$.  By Claim 2,
\[
t_{2l-1} = \mu \alpha (s_l , s_{l+1} , \dots , s_{2r} , \mu^{-1} s_0 , \dots , \mu^{-1} s_{l-1} ).
\]
Using the definition of $\alpha$, we get that
\begin{eqnarray*}
t_{2l-1} & = & \mu ( 2 s_{l+r}^2 + 4 ( s_l \mu^{-1} s_{l-1} + s_{l+1} \mu^{-1} s_{l-2} + \dots + s_{2l-1} \mu^{-1} s_0 \\
& + & s_{ 2l } s_{ 2r } + \dots + s_{ l+r-1 } s_{ l+r +1 } ) ) \\
& = & 2 \mu s_{l+r}^2 + 4 \sum_{ \{ i,j \} \in U_{2l-1} } \mu_{ \{ i,j \} } s_i s_j.
\end{eqnarray*}
Assume now that $0 \leq l \leq r -1$.  By Claim 2,
\[
t_{2l} = \mu^2 \alpha ( s_{l+r + 1} , s_{l+r + 2} , \dots , s_{2r} , \mu^{-1} s_0 , \dots , \mu^{-1} s_{l+r} ).
\]
We can now proceed as before using the definition of $\alpha$.
\end{proof}

\medskip

\noindent
\textbf{Claim 4}: Let $s,x \in \mathbb{F}_{q^t}$.  If $0 \leq l \leq 2r$, then the coefficient of
$\theta^l$ in $(s+x)^2$ is
\[
( s_{l/2} + x_{l/2} )^2 + 2 \sum_{ \{i,j \} \in U_l } \mu_{ \{i,j \} } (s_i + x_i )(s_j + x_j) ~\mbox{if $l$ is even},
\]
and
\[
\mu ( s_{ r + l/2 + 1/2} + x_{r + l/2 + 1/2} )^2 + 2 \sum_{ \{i,j \} \in U_l } \mu_{ \{i,j \} } (s_i + x_i )(s_j + x_j )
~ \mbox{if $l$ is odd}.
\]

\smallskip

\begin{proof}[Proof of Claim 4]
Consider
\[
(s + x)^2 = \sum_{i=0}^{2r} (s_i + x_i)^2 \theta^{2i} + 2 \sum_{0 \leq i < j \leq 2r} (s_i + x_i )(s_j + x_j) \theta^{i+j}.
\]
The claim follows from the definitions of $\mu_{ \{i,j \} }$, $U_l$, and the identity
$\theta^{2r + k } = \mu \theta^{k-1}$ for $1 \leq k \leq 2r$.
\end{proof}

\medskip

\noindent
\textbf{Claim 5}: If $(s,t) , (x,y) \in Z$ and $(s + x)^2 = t + y$, then
\[
\mu ( s_{l + r} - x_{l+r} )^2 + 2 \sum_{ \{i , j \} \in U_{2l-1} } \mu_{ \{i,j \} } (s_i - x_i  )(s_j  - x_j)= 0
 ~ \mbox{for $1 \leq l \leq r$},
\]
and
\[
(s_l - x_l)^2 + 2 \sum_{ \{i, j \} \in U_{2l} } \mu_{ \{ i , j  \} } (s_i - x_i )(s_j - x_j ) = 0
~ \mbox{for $0 \leq l \leq r $}.
\]

\smallskip
\begin{proof}[Proof of Claim 5]
By Claim 4, equating coefficients of $1, \theta , \dots , \theta^{2r}$ in the equation
$(s + x)^2 = t + y$ gives
\[
t_{2l-1}  + y_{2l-1} = \mu ( s_{l+r} + x_{l+r} )^2 + 2 \sum_{ \{i,j \} \in U_{2l-1} } \mu_{ \{i,j \} } (s_i + x_i )(s_j + x_j)
 ~ \mbox{if $1 \leq l \leq r$},
\]
and
\[
t_{2l}  + y_{2l} = (s_l + x_l)^2 + 2 \sum_{ \{ i , j \} \in U_{2l} } \mu_{ \{i,j \} } (s_i + x_i)(s_j + x_j)
~ \mbox{if $0 \leq l \leq r$}.
\]
Now we apply Claim 3 to $t_{2l-1}$ and $y_{2l-1}$.  This gives
\begin{eqnarray*}
2 \mu ( s_{l+r}^2 +  x_{l+r}^2) & + & 4 \sum_{ \{i,j \} \in U_{2l-1} } \mu_{ \{i,j \} } (s_i s_j + x_i x_j)   =  \\
\mu ( s_{l+r}  +  x_{l+r} )^2 &+& 2 \sum_{ \{i,j \} \in U_{2l-1} } \mu_{ \{i,j \} } (s_i + x_i )(s_j + x_j)
\end{eqnarray*}
for $1 \leq l \leq r$.  This can be rewritten as
\[
\mu ( s_{l + r} - x_{l+r} )^2 + 2 \sum_{ \{i , j \} \in U_{2l-1} } \mu_{ \{i,j \} } (s_i - x_i  )(s_j  - x_j)= 0.
\]
A similar application of Claim 3 (and Claim 1 in the case of $t_{2r}$ and $y_{2r}$) gives
\[
(s_l - x_l)^2 + 2 \sum_{ \{i, j \} \in U_{2l} } \mu_{ \{ i , j  \} } (s_i - x_i )(s_j - x_j ) = 0
\]
for $0 \leq l \leq r$.
\end{proof}

\medskip

We are now ready to find an upper bound on the maximum degree of the subgraph of $G_{q^t}$ induced by
$Z$.  Fix a vertex $(s,t) \in Z$.  Suppose $(x,y)$ is a neighbor of $(s,t)$ with $(x,y) \in Z$.  By Claim 5,
$(x_0 , \dots , x_{2r} ) \in \mathbb{F}_q^{2r+1}$ is a solution to the system
\[
\mu ( s_{l + r} - x_{l+r} )^2 + 2 \sum_{ \{i , j \} \in U_{2l-1} } \mu_{ \{i,j \} } (s_i - x_i  )(s_j  - x_j)= 0
 ~ \mbox{for $1 \leq l \leq r$},
\]
and
\[
(s_l - x_l)^2 + 2 \sum_{ \{i, j \} \in U_{2l} } \mu_{ \{ i , j  \} } (s_i - x_i )(s_j - x_j ) = 0
~ \mbox{for $0 \leq l \leq r $}.
\]
If we set $z_i = s_i - x_i$ for $0 \leq i \leq 2r$, then we see that we have a solution to the following system of
$2r+1$ homogeneous quadratic equations in the $2r+1$ unknowns $z_0 , \dots , z_{2r+1}$:
\[
\mu  z_{l+r}^2 + 2 \sum_{ \{i , j \} \in U_{2l-1} } \mu_{ \{i,j \} } z_i z_j = 0
 ~ \mbox{for $1 \leq l \leq r$},
\]
and
\[
z_l^2 + 2 \sum_{ \{i, j \} \in U_{2l} } \mu_{ \{ i , j  \} }  z_i z_j  = 0
~ \mbox{for $0 \leq l \leq r $}.
\]

\begin{lemma}\label{le:sams lemma}
The number of solutions $(z_0 , \dots , z_{2r} ) \in \fq^{2r+1}$ to the above system of
$2r+1$ homogeneous quadratic equations is at most
\[
\frac{2r+5}{3} q^{ \frac{4t}{3} +1 }.
\]
\end{lemma}
\begin{proof}
Let $m$ be the largest integer such that $m \leq \frac{ 2 ( r + 1) }{3}$.
 %The number of solutions in which at least
%$m$ of the $z_i$'s are 0 is at most $\binom{2r+1}{m}q^{2r + 1 - m }$.
%For the rest of the lemma, we restrict to solutions $(z_0 , \dots , z_{2r} )$ in which at most $m-1$ of the $z_i$'s are 0.
 If there is set
$T$ of size $m$ such that each $z_i$' in the set $\{ z_i : i \in T \}$  either is zero or is determined uniquely by the
$z_j$'s in the set $\{ z_j : j \in \{0 , 1 , \dots , 2r \} \backslash T \}$, then  there are at most
$q^{2r + 1 - m }$  solutions. We will show there are at most $(m+1)$ choices for the index set $T$ which implies that we have at most $(m+1)q^{2r+1-m}$ solutions in total.

Let $A$ be the $2( r - m + 1) \times m $ matrix with entries in $\fq$,  where for
$1 \leq i \leq 2 (r - m + 1)$ and $1 \leq j \leq m$, the $(i,j)$ entry of $A$ is the coefficient of
$z_{j-1}$ in the equation
\[
\mu z_{ m + r + ( i - 1)/2 }^2 + 2 \sum_{ \{k,l \} \in U_{ 2( m + (i-1)/2) -1 } } \mu_{ \{k,l \} } z_k z_l = 0
\]
if $i$ is odd, and the coefficient of $z_{j-1}$ in the equation
\[
z_{ m + ( i - 2)/2 }^2 + 2 \sum_{ \{ k , l \} \in U_{ 2( m + (i - 2)/2) } } \mu_{ \{k,l \} } z_k z_l = 0
\]
if $i$ is even.  If $A$ is the matrix formed in this way, then one check that
 \[
 A=
 \begin{pmatrix}
 z_{2m-1} & z_{2m-2} & z_{2m - 3} & \ldots & z_{m} \\
 z_{2m} & z_{2m-1} & z_{2m - 2} & \ldots & z_{m+1} \\
 z_{2m+1} & z_{2m} & z_{2m - 1} & \ldots & z_{m+2} \\
\vdots   & \vdots & \vdots & \ddots & \vdots\\
 z_{2r} & z_{2r-1} & z_{2r - 2} & \ldots & z_{2r-m+1} \\
 \end{pmatrix}
 \]
As $m \leq \tfrac{2(r+1)}{3}$, we have $2(r-m+1) \geq m$.
 Let $x = ( z_0 , z_1 , \dots , z_{m-1} )^T$ and
\[
b_i =
\left\{
\begin{array}{ll}
- 2^{-1} \mu z_{ m + r +   (i-1)/2  }-\sum_{\underset{\{k,l\} \cap \{0,1,\ldots,m-1\}=\emptyset} {\{k,l \} \in U_{ 2( m + (i-1)/2)-1  }} } \mu_{ \{k,l \} } z_k z_l & \mbox{if $i$ is odd} \\
- 2^{-1}  z_{ m + (i-2)/2 }-\sum_{\underset{\{k,l\} \cap \{0,1,\ldots,m-1\}=\emptyset} {\{k,l \} \in U_{ 2( m + (i-2)/2) }} } \mu_{ \{k,l \} } z_k z_l & \mbox{if $i$ is even}
\end{array}
\right.
\]
for $1 \leq i \leq 2(r - m + 1)$.  Let $b = ( b_1 , b_2 , \dots , b_{2 ( r- m + 1) } )^T$.

We next show an upper bound $m+1$ for the possible choices for the index set $T$.
% If we can show that under the assumption that $z_m , \dots , z_{2r}$ are fixed elements of $\fq$, there is a unique solution $x$ to the equation $Ax %= b$ then the proof will be complete as this shows $z_m , \dots , z_{2r}$ determine $z_0 , \dots , z_m$.
 Let $r_i$ be the $i$-th row of $A$ so that
\[
r_i = ( z_{ 2m - 1 + (i - 1) } , z_{2m - 2 + ( i - 1) } , z_{2m - 3 + (i-1) } , \dots , z_{m + ( i - 1) } ).
\]
If $r_1={\bf 0}$, then we take $T=\{m,m+1,\ldots,2m-1\}$. We assume $r_1 \not = {\bf 0}$ for the rest of the proof of the lemma.  

\smallskip

\noindent
\textbf{Claim}: If $i \geq 1$ and $r_{i+1} \in \textup{Span}_{ \fq} \{ r_1 , \dots , r_i \}$, then
\[
z_{m + i + j } \in \textup{Span}_{ \fq} \{ z_m , z_{m+1} , \dots , z_{m + i  - 1} \}~ \mbox{for}~ j=0,1, \dots , m - 1.
\]

\smallskip

\begin{proof}[Proof of Claim]
We prove the claim by induction on $j$.  Suppose
\begin{equation}\label{eq:10.1}
r_{i+1} = \alpha_1 r_1 + \dots + \alpha_i r_i
\end{equation}
for some $\alpha_j \in \fq$.  By considering the last coordinate, we get that
\[
z_{m + i } = \sum_{j=1}^{i} \alpha_j z_{ m + ( j - 1) } \in \textup{Span}_{ \fq} \{ z_m , \dots , z_{m + i - 1} \}
\]
establishing the base case $j=0$.  If $z_{m + i + j_0 } \in \textup{Span}_{ \fq} \{ z_m , \dots , z_{m + i - 1 } \}$ for
$0 \leq j_0 \leq m -2$, then by (\ref{eq:10.1}),
\begin{equation*}
z_{m + i + j_0 + 1}  =  \sum_{j=1}^{i} \alpha_j z_{ m + j_0 + 1 + (j - 1) }
 =  \alpha_i z_{m + i + j_0  } + \sum_{j=1}^{i-1} \alpha_j z_{m +  j + j_0 }.
\end{equation*}
By the inductive hypothesis, this is in $\textup{Span}_{ \fq} \{ z_m , \dots , z_{m + i - 1} \}$.
\end{proof}

\smallskip

By the Claim, if there is an $i \in \{1 , 2, \dots , m - 1 \}$  such that
$r_{i+1} \in \textup{Span}_{ \fq} \{ r_1 , \dots , r_i \}$, then there exist $m$~$z_i$'s that are uniquely determined by the
other $z_j$'s and we can take $T=\{z_{m+i},z_{m+i+1},\ldots,z_{2m+i-1}\}$.
  Otherwise, $r_1 , \dots , r_m$ are linearly independent which implies that the rank of $A$ is at least $m$.  It is at this step where we need $m \leq \frac{2 (  r+ 1) }{3}$ as we
require the number of rows of $A$, which is $2( r - m + 1)$, to be at least $m$.
Since the rank of $A$ is at least $m$, there is at most one solution $x$ to $Ax = b$.  In this case we take $T=\{0,1,\ldots,m-1\}$
and each of  $\{z_0,z_1,\ldots,z_{m-1}\}$ is determined by $\{z_m,z_{m+1},\ldots,z_{2r}\}$.

Altogether, we have at most
\[
(m+1) q^{ \frac{4r}{3} - \frac{1}{3} } \leq \frac{2r+5}{3} q^{ \frac{4r}{3} +1 }
\]
solutions which proves the lemma.

\end{proof}

\smallskip

By Lemma \ref{le:sams lemma}, the maximum degree of $G_{q^t} [Z]$ is at most
$\tfrac{2r+5}{3} q^{ \tfrac{4r}{3} +1 }$.  Therefore,
\[
\chi ( G_{q^t} ) \leq  \frac{2r+5}{3} q^{ \frac{4r}{3} +1 } + (2r + 1)q^{r + 1}
\]
\end{proof}

\smallskip

We obtain a coloring of $ER_{q^{t}}$ from a coloring of $G_{q^t}$ as before.  We use one new color on the vertices
$z_1, \dots , z_t$, and then give $y$ any color that is used on $G_{q^t}$.  This gives a coloring of $ER_{q^t}$ that uses at most
$ \frac{2r+5}{3} q^{ \frac{4r}{3} +1 }+(2r+1) q^{r+1} + 1$ colors which proves
Theorem \ref{th:chromatic number 2}.

%%%%%%%%%%%%%%%%%%%%%%%%%%%%%%%%%%%%%%

\section{Proof of Theorem \ref{th:main theorem}}

The following lemma is easily proved using the definition of adjacency in $G_q$.

\begin{lemma}
Suppose $\alpha_i , \alpha_j , \alpha_k$ are distinct elements of $\mathbb{F}_q$ such that
\[
\alpha_i + \alpha_j = a^2, \alpha_j + \alpha_k = b^2, ~\mbox{and}~ \alpha_k + \alpha_i = c^2
\]
for some $a,b,c \in \mathbb{F}_q$.  If $x + y = a$, $y+ z = b$, and $z + x = c$, then
$\{ ( x , \alpha_i ) , (y , \alpha_j) , (z , \alpha_k)\}$ induces a triangle in $G_q$.
\end{lemma}

Given an odd prime power $q$, let $\chi : \mathbb{F}_q \rightarrow \{ 0 , \pm 1 \}$ be the quadratic character on $\mathbb{F}_q$.  That is, $\chi (0 ) = 0$, $\chi (a ) =1$ if $a$ is a nonzero square in $\mathbb{F}_q$, and $\chi (a) = -1$ otherwise.
For the next lemma, we require some results on finite fields (see Chapter 5 of \cite{LN}).

\begin{proposition}\label{char} Let $q$ be an odd prime and $f(x) = a_2 x^2 + a_1 x + a_0 \in \mathbb{F}_q [x]$ where $a_2 \neq 0$.
If $a_1^2 - 4a_0 a_2 \neq 0$, then
\[
\sum_{c \in \mathbb{F}_q} \chi (f(c)) = - \chi(a_2).
\]
\end{proposition}

\begin{proposition}[Weil]\label{weil}
If $f (x) \in \mathbb{F}_q [x] $ is a degree $d \geq 1$ polynomial that is not the square of another polynomial, then
\[
\left| \sum_{x \in \mathbb{F}_q } \chi (f (x) ) \right| \leq (d - 1 ) \sqrt{q}.
\]
\end{proposition}

\begin{lemma}\label{choice}
If $q > 487$ is a power of an odd prime, then there are elements $\alpha_1 , \dots , \alpha_5 \in \mathbb{F}_q^*$ such that
$ \alpha_1 , \dots ,  \alpha_5$ are all distinct, and
\[
\chi ( \alpha_i + \alpha_j ) = 1
\]
for $1 \leq i < j \leq 5$.
\end{lemma}
\begin{proof}
Choose $\alpha_1 \in \mathbb{F}_q^*$ arbitrarily.
There are $\frac{q-1}{2}$ nonzero squares in $\mathbb{F}_q^*$ so we can easily
find an $\alpha_2 \in \mathbb{F}_q \backslash \{ 0 ,  \alpha_1 \}$ such that
$\chi ( \alpha_1 + \alpha_2 ) = 1$.  Observe that this implies $\alpha_2 \neq - \alpha_1$ otherwise
$\chi ( \alpha_1 + \alpha_2 ) = 0$.  Assume that we have chosen
$\alpha_1 , \dots , \alpha_k \in \mathbb{F}_q^*$ so that
$  \alpha_1 , \dots ,  \alpha_k$ are all distinct and
\[
\chi( \alpha_i + \alpha_j ) = 1
\]
for $1 \leq i < j \leq k$.  Let
\[
f(x) = \prod_{i = 1}^{k} ( 1 + \chi ( \alpha_i + x) )
\]
and $X = \{ \beta \in \mathbb{F}_q : f( \beta ) = 2^k \}$.  If $\beta \in \mathbb{F}_q$ and $f( \beta ) > 0$, then
$\chi( \alpha_ i + \beta ) \in \{0 , 1 \}$ for $1 \leq i \leq k$.  We have $\chi ( \alpha_i  + \beta) = 0$ if and only if
$\beta =  - \alpha_i$.  Therefore, there are at most $k$ distinct $\beta$'s in
$\mathbb{F}_q$ such that $0 < f ( \beta ) < 2^k$ which implies
\begin{equation*}
2^k |X| + k 2^{k-1}  \geq  \sum_{x \in \mathbb{F}_q} f(x) =
q + \sum_{\emptyset \neq S \subset [k] }\sum_{x \in \mathbb{F}_q} \chi \left( \prod_{ \alpha \in S}( x + \alpha ) \right) .
\end{equation*}
For any $i \in [k]$, $\sum_{x \in \mathbb{F}_q} \chi( \alpha_i + x) = 0$.  For any $1 \leq i < j \leq k$,
\[
\sum_{x \in \mathbb{F}_q} \chi( x^2 + ( \alpha_i + \alpha_j ) x + \alpha_i \alpha_j ) = -1
\]
by Proposition \ref{char}.  When $k = 2$, if $\frac{q - 3  - 3 \cdot 2^2 }{2^3} \geq 3 + 1$, then
$|X| \geq 4$ and we can choose an $\alpha_3 \in \mathbb{F}_q \backslash \{ 0 , \alpha_1 , \alpha_2 \}$ that has the desired properties.  When $k \in \{3,4 \}$, we use Weyl's inequality to obtain a lower bound on $|X|$.  Observe that
since $\alpha_1 , \dots , \alpha_k$ are distinct, no product $(x + \alpha_{i_1} ) \cdots (x + \alpha_{i_l} )$
for $1 \leq i_1 < \dots < i_l \leq k$ is the square of a polynomial in $\fq [x]$.  By Weyl's inequality, for any
$1 \leq i_1 < \dots < i_l \leq k$,
\[
\left| \sum_{x \in \mathbb{F}_q} \chi( (x + \alpha_{i_1} ) \cdots (x + \alpha_{i_l} ) ) \right| \leq l \sqrt{q}.
\]
The hypothesis on $q$ implies that $q$ is large enough so that $|X| \geq 5$ when $k=3$, and $|X| \geq 6$ when $k = 4$.   Therefore, we can inductively choose $\alpha_4$ and $\alpha_5$ so that
$\alpha_1 , \dots , \alpha_5$ satisfy all of the required properties.
\end{proof}

\smallskip

Choose elements $\alpha_1 , \dots , \alpha_5 \in \mathbb{F}_q^*$ satisfying the properties of Lemma \ref{choice}.
Let
\[
\alpha_i  + \alpha_j = a_{i,j}^2
\]
for $1 \leq i < j \leq 5$.  Since $\alpha_i \neq - \alpha_j$, no $a_{i,j}$'s is zero.  For $1 \leq i < j < k \leq 5$,
let $x_{i,j,k}$, $y_{i,j,k}$, and $z_{i,j,k}$ be any elements of $\mathbb{F}_q$ that satisfy
\begin{center}
$x_{i,j,k} + y_{i,j,k} = a_{i,j}$, ~~ $y_{i,j,k} + z_{i,j,k} = a_{j,k}$, ~ and~ $z_{i,j,k} + x_{i,j,k} = a_{i,k}$.
\end{center}
Then the vertices
$( x_{i,j,k} , \alpha_i )$, $(y_{i,j,k} , \alpha_j )$, and $(z_{i,j,k} , \alpha_k )$ form a triangle in $G_q$
and this holds for any $1 \leq i < j < k \leq 5$.

Now we use these triangles, which are in $G_q$, together with the new vertices $z_1 , \dots , z_q , y$ that are added to $G_q$ to from $H_q \cong ER_q$ to obtain a subgraph with chromatic number at least four.
For $1 \leq i \leq 5$, the vertex $z_{ \alpha_i}$ is adjacent to all vertices of the form
$(x , \alpha_i )$.  The vertex $y$ is adjacent to each $z_i$.  Consider the subgraph $H_q$ whose vertices are
$y , z_{ \alpha _1} , \dots , z_{ \alpha_5}$ together with all $(x_{i,j,k} , \alpha_i )$, $(y_{i,j,k} , \alpha_j )$, and
$(z_{i,j,k} , \alpha_k)$ for $1 \leq i < j < k \leq 5$.  Suppose we have a proper 3-coloring of this subgraph, say with colors 1, 2, and 3.  If the color 1 is given to three distinct vertices $z_{ \alpha_i}$'s, say
$z_{ \alpha_i } , z_{ \alpha_j}$, and $z_{ \alpha_k}$, then only colors 2 and 3 may be used on the triangle
whose vertices are $( x_{i,j,k} , \alpha_i )$, $(y_{i,j,k} , \alpha_j )$, and $(z_{i,j,k} , \alpha_k )$.  Therefore, all three colors must be used to color the vertices in the set $\{ z_{ \alpha_1 } , \dots , z_{ \alpha_5} \}$ but then no color may be used on $y$.
The number of vertices in this subgraph is at most $1 + 5 + \binom{5}{3}3 =36$.

%%%%%%%%%%%%%%%%%%%%%%%%%%%%%%%%%%%%%%%%%%%%%%%%%%

\section{Concluding Remarks}

The upper bounds of Theorems \ref{th:gq coloring} and \ref{th:chromatic number 2} can be improved for large $q$ by applying the result of Alon et.\ al.\ \cite{aks} to the graph $G_{q^t} [Z]$ instead of using Brooks' Theorem.  We have chosen to use Brooks' Theorem as then there is no issue of how large the implicit constant is, and because we believe that the upper bound
should be closer to $O( q^{t/2} )$.

Using a similar argument as the one used to prove Theorem \ref{th:chromatic number 2}, we can prove the following.

\begin{theorem}
If $q$ is a power of an odd prime, then
\[
\chi ( G_{q^3} ) \leq 6 q^2.
\]
\end{theorem}

A consequence is that $\chi( ER_{q^3} ) \leq 6q^2 + 1$ whenever $q$ is a power of an odd prime.  Unfortunately, we were not able to extend this bound to the general case.  In the $q^3$ case, one can explicitly compute the relations satisfied by vertices in
$X$ (see Section 3.2) and then use these equations to bound the maximum degree of $G_{q^3}[X]$.  Dealing with the set $X$ is one of the main obstacles in our approach.

We remark that if one could improve the bound in Lemma \ref{le:sams lemma}, it would improve our result in Theorem \ref{th:chromatic number 2}. It seems likely that Lemma \ref{le:sams lemma} could be strengthened enough to improve the bound in Theorem \ref{th:chromatic number 2} all the way down to
\[
\chi(ER_{q^{2r+1}}) \leq (2r+1+o(1)) q^{r+1}.
\]
Perhaps this can be done using techniques from algebraic geometry.

\medskip

The conditions on $q$ for which there exists an irreducible polynomial $x^t - \mu \in \fq [x]$ are known (see Theorem 3.75 of \cite{LN}).  Let $q$ be a power of an odd prime and let $t  \geq 3$ be an odd integer.  Let
$\textup{ord}( \mu , q)$ be the order of $\mu$ in group $\fq^*$.  Then $x^t - \mu \in \fq [x]$ is irreducible if and only if
each prime factor of $t$ divides $\textup{ord}( \mu , q)$ but does not divide $\frac{ q- 1}{ \textrm{ord}( \mu , q) }$.
Since $\fq^*$ is cyclic, for any divisor $d$ of $q-1$, there is an element $a \in \fq^*$ with $\textup{ord}(a , q ) = d$.
As long as $t$ divides $q - 1$, we can choose an element $\mu \in \fq^*$ so that
$t$ divides $\textup{ord}( \mu , q)$ but $t$ does not divide $\frac{q-1}{ \textup{ord} ( \mu , q)}$.  Therefore,
if $q \equiv 1 ( \textup{mod}~t)$, then Theorem \ref{th:chromatic number 2} applies and we obtain the upper bound
$\chi ( ER_{q^t} ) \leq \frac{2r+5}{3} q^{ \frac{4r}{3} +1 }+(2r+1)q^{r+1}+1$ in this case.
For a fixed $t \geq 3$,  Dirichlet's Theorem on primes in arithmetic progressions
 implies that there are infinitely many primes $p$ such that
$p \equiv 1 ( \textup{mod}~t)$. Then for any $q$ which is an odd power of such $p$, we have $q \equiv 1
( \textup{mod}~t)$.
\medskip

The graph $ER_q$ is an example of an orthogonal polarity graph.  Let $\Pi$ be a finite projective plane of order $q$ with
point set $\mathcal{P}$ and line set $\mathcal{L}$.  A \emph{polarity} of $\Pi$ is a bijection
$\pi : \mathcal{P} \cup \mathcal{L} \rightarrow \mathcal{P} \cup \mathcal{L}$ such that
$\phi ( \mathcal{P} ) = \mathcal{L}$, $\phi ( \mathcal{L} ) = \mathcal{P}$, $\phi^2$ is the identity, and point $p$ is on line $l$ if and only if point $\phi (l)$ is on line $\phi (p)$.  A point $p$ for which $p \in \phi (p)$ is called an \emph{absolute point}.  The
polarity $\phi$ is called \emph{orthogonal} if it has exactly $q +1$ absolute points.  The corresponding orthogonal polarity graph is the graph with vertex set $\mathcal{P}$, and vertices $p_1$ and $p_2$ are adjacent if and only if $p_1 \in \phi (p_2)$.
The graph $ER_q$ is an orthogonal polarity graph which can be obtained from the projective plane $PG(2,q)$ and the polarity that sends the point $(x,y,z)$ to the line $[x,y,z]$, and sends the line $[x,y,z]$ to the point $(x,y,z)$ (loops in this graph must be removed to obtain $ER_q$).  There are non-desaurgesian projective planes that have orthogonal polarities and this leads us to the following problem.

\begin{problem}
Determine if there is an absolute constant $C$ such that the following holds.
If $G$ is an orthogonal polarity graph of a projective plane of order $q$, then
\[
\chi (G) \leq C q^{1/2}.
\]
\end{problem}

%%%%%%%%%%%%%%%%%%%%%%%%%%%%%%%%%%%%%%%%%%%%%%%

\end{document}